\documentclass[12pt]{amsart}
\usepackage{amssymb,amsthm,amsmath,amsfonts}
\usepackage[usenames,dvipsnames]{xcolor}
\usepackage{graphicx}
\usepackage{epstopdf}
\usepackage[a4paper,margin=2cm]{geometry}

\def\R{{\mathbb R}}

\def\<{\langle}
\def\>{\rangle}

\def\E{{\mathbb E}}

\newcommand{\be}{\begin{equation}}
\newcommand{\ee}{\end{equation}}
\newtheorem{theorem}{Theorem}[section]
\newtheorem{proposition}[theorem]{Proposition}

\newtheorem{definition}[theorem]{Definition}

\theoremstyle{remark}
\newtheorem{remark}[theorem]{Remark}

\title[Boolean cumulants and the free Lukacs property]{Conditional expectations\\ through Boolean cumulants and subordination  \\
 -	towards a better understanding of \\ the Lukacs property in free 
	probability}
\author[K. Szpojankowski]{Kamil Szpojankowski}
\author[J. Weso\l{}owski]{Jacek Weso\l{}owski}
\address{
Wydzia\l{} Matematyki i Nauk Informacyjnych\\
Politechnika Warszawska\\
ul. Koszykowa 75\\
00-662 Warszawa, Poland.}
\email{k.szpojankowski@mini.pw.edu.pl, j.wesolowski@mini.pw.edu.pl}
\thanks{Research partially supported by the NCN (National Science
	Center) grant 2016/21/B/ST1/00005}


\begin{document}
\begin{abstract}
	 Following recently discovered connections between Boolean 
	 cumulants and freeness, we use them to derive explicit formulas for a 
	 family of conditional expectations in free variables. Further, we show 
	 how the approach through Boolean cumulants together with subordination 
	 simplifies some Lukacs--type  regression characterizations in free 
	 probability. Finally, we explain how the free dual Lukacs property can 
	 be used to get a pocket proof of the free version of the direct Lukacs 
	 property.
\end{abstract}
\maketitle
\section{Introduction}
	
	In this paper we study characterizations of probability measures in terms of free random variables. Problems of a similar type were studied for long time in classical probability. Among many examples the most prominent one is the Kac-Bernstein theorem which states that for independent random variables $X,Y$, the random vector $(U,V)=(X+Y,X-Y)$ has independent components if and only if $X$ and $Y$ have Gaussian distribution with the same variance. Another imortant example is the Lukacs theorem which states that for independent random variables $X,Y$,  the random vector $(U,V)=(X/(X+Y),X+Y)$ has independent components if and only if $X$ and $Y$ have Gamma distributions with the same scale parameter (cf. \cite{Lukacs}). 
	
	 Throughout recent years (mostly already in XXI century) it has been observed that many classical characterizations of probability measures have their counterparts in the framework of free probability. The analogue of the Kac-Bernstein theorem was studied by Nica (see \cite{NicaChar}) and states that for free $X,Y$, the random variables $U=X+Y$ and $V=X-Y$ are free if and only if $X,Y$ have Wigner semicircular distribution with the same variance. The free analogue of the Lukacs theorem was studied in \cite{SzpLukacsProp}. 
	
	It was also observed that the strong assumption that both vectors $(X,Y)$ and $(U,V)$ consist of independent (respectively free) random variables can be weakened and it is enough to assume only that some conditional moments of $U$ given $V$ are scalar multiples of the
	unit. This phenomenon was observed both in context of commutative, independent random variables (see \cite{LahaLukacs,BobWes2002Dual,WesolGamma}), as well as for non-commutative, free random variables (c.f. \cite{BoBr2006,SzpojanWesol}). On the other hand calculation of conditional moments of functions of non-commutative random variables typically is highly non-trivial. In \cite{EjsmontFranzSzpojankowski:2015:convolution} the subordination methodology of free convolutions was applied quite naturally to some characterization problems. Its main advantage is a considerable simplification of the proofs. In this research our goal originally was to apply the subordination methodology to characterizations through free dual Lukacs regressions in order to  simplify rather complex proofs from \cite{SzpDLRNeg}. The approach proved to be useful for one of the regression characterizations from \cite{SzpDLRNeg}. However, as powerful as it is, subordination does not cover the second of the dual Lukacs--type regression of negative order considered in \cite{SzpDLRNeg}. The problem lies in lack of an explicit expression for a rather complicated conditional expectation appearing in the regression condition. 	Recently discovered connections between Boolean cumulants \cite{FMNS2,LehnerSzpojan} and free probability allow to overcome that difficulty. In particular it was observed in these references that Boolean cumulants appear quite naturally in calculations of conditional expectations of some functions of free random variables. We will apply some of the results from  \cite{LehnerSzpojan} to derive an explicit expression for the conditional expectation we are interested here in. One of our main points is that results from the present paper and \cite{LehnerSzpojan} show that Boolean cumulants are natural tool to calculate explicitly conditional expectations in free probability.

	More precisely, we present a new approach to problems considered in \cite{SzpDLRNeg}, i.e. we work in dual scheme to the Lukacs theorem (observe that in the framework of the classical Lukacs theorem one has $X=UV$ and $Y=V(1-U)$), we assume that some conditional moments of $V^{1/2}(1-U)V^{1/2}$ given $V^{1/2}UV^{1/2}$ are scalar multiples of the unit, and conclude that $U$ has a free Binomial distribution and $V$ has a free Poisson distribution. As we mentioned above, it turns out, that the characterizations considered in \cite{SzpDLRNeg} are not a straightforward application of the subordination technique from \cite{EjsmontFranzSzpojankowski:2015:convolution}. 
	
	The main technical result of the present paper is an explicit formula for conditional expectation 
	\begin{align}\label{dualLukacsCondExp}
	\E_V\left((1-U)^{-1}U^{1/2}zU^{1/2}VU^{1/2} (1-zU^{1/2}VU^{1/2})^{-1} U^{1/2}(1-U)^{-1}\right),
	\end{align}
	where $\E_V$ denotes the conditional expectation on the algebra generated by V and random variables $U,V$ are free. Actually, we will be interested in a more general family of conditional expectations which, we believe, will be of  a wider interest, e.g. in other regression characterization problems in free probability.  Next we present easier proofs of main results of \cite{SzpDLRNeg}, for one of them we apply just the subordination technique alone and for the other the subordination is enhanced by the Boolean cumulants technique which enables to derive an explicit form of \eqref{dualLukacsCondExp}.
	
	We also present a surprisingly compact proof of the free Lukacs property for free Poisson distributed random variables. That is, we show that for free random variables $X$ and $Y$ both having free Poisson (Marchenko-Pastur) distribution, the random variables $U=X+Y$ and $V=(X+Y)^{-1/2}X(X+Y)^{-1/2}$ are free. In \cite{SzpLukacsProp} we presented a ''hands on'', direct combinatorial proof. There the strategy was to show that all mixed free cumulants of $U$ and $V$ vanish. In particular, the explicit formula for joint free cumulants of $X,X^{-1}$ (of an arbitrary order) for invertible, free Poisson distributed random variable $X$ was derived there. Here we show that the direct free Lukacs property actually follows from its dual version proved in \cite{SzpojanWesol}.
	
	Except from Introduction this paper has 4 more sections. In Section 2 we set up the framework and recall necessary notions and results. Section 3 is devoted to derivation (through Boolean cumulants)  of explicit expressions for a family of conditional expectations which contains the conditional expectation \eqref{dualLukacsCondExp}, as a prominent member. In Section 4 we present how to use subordination and Boolean cumulants (results of Sect. 3) to simplify proofs of free dual Lukacs regression characterizations from \cite{SzpDLRNeg}. Section 5 contains  a pocket proof of the free version of the direct 
	Lukacs property (based on the dual one).
\section{Notation and background}
	
	In this section we introduce notions and results from non-commutative probability. We restrict the background to essential facts which are necessary in the subsequent sections. Readers, who are not familiar with non-commutative (in particular, free) probability, to get a wider perspective, may choose to consult one of the books \cite{MingoSpeicher,NicaSpeicherLect}. 
	
	We assume that $\mathcal{A}$ is a unital $*$-algebra and $\varphi:\mathcal{A}\mapsto\mathbb{C}$ is a linear functional which is normalized (that is, $\varphi\left(1_{\mathcal{A}}\right)=1$, where $1_{\mathcal{A}}$ is a unit of $\mathcal{A}$), positive, tracial and faithful.  We will refer to the pair $(\mathcal{A},\varphi)$ as a non-commutative probability space.
	
	\subsection{Freeness, free and Boolean cumulants}
	
	The concept of freeness was introduced by Voiculescu in \cite{VoiculescuAdd} and among several existing notions of non-commutative independence is the most prominent one. Here we recall its definition.
	\begin{definition} Let  $(\mathcal{A},\varphi)$ be a non-commutative probability space. 
		We say that subalgebras $\left(\mathcal{A}_i\right)_{1\le i\le n}$ of algebra $\mathcal{A}$ are free if for any choice of $X_k\in \mathcal{A}_{i_k}$ which is centered, i.e. $\varphi\left(X_k\right)=0$, $k=1,\ldots,n$,
		\begin{align*}
		\varphi\left(X_1\cdots X_n\right)=0
		\end{align*}
		whenever  neighbouring random variables come from different subalgebras, that is when $i_k\neq i_{k+1}$ for all $k=1,\ldots,n$, where $i_{n+1}:=i_1$.
	\end{definition}

	It turns out that freeness has a nice combinatorial description which uses the lattice of non-crossing partitions. 
	\begin{definition} $\ $
		\begin{enumerate} 
		\item
		%By a partition of a finite totally ordered set $S$ we understand a family of subsets of $B_1,\ldots,B_k\subseteq S$ such that $\bigcup_{j=1}^k B_j=S$ where $B_j$ are non-empty and pairwise disjoint. 
		For a positive integer $n$ denote $[n]:=\{1,\ldots,n\}$. A partition $\pi$ of $[n]$ is a collection of non-empty, pair-wise disjont subsets  $B_1,\ldots,B_k\subseteq [n]$ such that $\bigcup_{j=1}^k B_j=[n]$. The subsets $B_j$ for $j=1,\ldots,k$ are called blocks of $\pi$, the number of blocks in $\pi$ is called the size of $\pi$ and is denoted by $|\pi|$, i.e. we have $|\pi|=k$.  
		
		\noindent
		The family of all partitions of  $[n]$ is denoted by  $\mathcal{P}(n)$. 	
		\item We say that $\pi\in \mathcal{P}(n)$ is a non-crossing partition if for any $B_1,B_2\in\pi$ and $1 \leq i_1<j_1<i_2<j_2\leq n$, $$\left(i_1,i_2\in B_1\quad\mbox{and}\quad j_1,j_2\in B_2\right)\quad \Rightarrow \quad B_1=B_2.$$ The family of all non-crossing partitions of $[n]$ is denoted by $NC(n)$.
		
		\item We say that $\pi\in \mathcal{P}(n)$ is an interval partition if for any $B_1,B_2\in\pi$ and $1 \leq i_1<j_1<i_2\leq n$   $$\left(i_1,i_2\in B_1\quad\mbox{and}\quad j_1\in B_2\right)\quad \Rightarrow \quad B_1=B_2.$$ The family of all interval partitions of $[n]$ is denoted by $Int(n)$.
	\end{enumerate}
	\end{definition}
	
It is useful to introduce a partial order $\leq$ on $NC(n)$ called the reversed refinement order.
	
	\begin{definition} $\ $
		For $\pi,\sigma\in \mathcal{P}(n)$ we say that that $\pi\leq \sigma$ if for any block $B\in\pi$ there exists a block $C\in\sigma$ such that $B\subseteq C$. The order $\leq$ is the reversed refinement order and it is also a partial order on the sets $NC(n)$ and $Int(n)$. 
		
		By $1_n$ we denote the maximal partition of $[n]$ with respect to $\leq$, i.e. the partition with one block equal to $[n]$.

%		On the set $NC(n)$ we will consider also another partial order denoted by $\ll$. For $\pi,\sigma\in NC(n)$ we say that $\pi\ll\sigma$ when $\pi\leq \sigma$ and for any block $C\in\sigma$ there is a block $B\in \pi$ such that $\min(C),\max(C)\in B$.
	\end{definition}
	
	It turns out that $(NC(n),\leq)$ and $(Int(n),\leq)$ have a lattice structure, for details we refer to \cite{NicaSpeicherLect}, Lectures 9 and 10.
	
%	We need an additional piece of structure on the lattice of non-crossing partitions, an anti-isomorphism called Kreweras complement introduced in \cite{Kreweras}. 
%	\begin{definition}
%		Given a noncrossing partition $\pi$ on $\{1,2,\dots,n\}$,
%		the Kreweras complement $K(\pi)$
%		is the $\le$-maximal 
%		noncrossing partition of the ordered set $\{\bar{1},\bar{2},\dots,\bar{n}\}$
%		such that $\pi\cup K(\pi)$ is a noncrossing partition on
%		$\{1,\bar{1},2,\bar{2},\dots,n,\bar{n}\}$.
%	\end{definition}
	
%	For a fixed non-commutative probability space $(\mathcal{A},\varphi)$ one introduces a multiplicative functional related to a partition in the following way
%	\begin{definition}
%		Let $\pi=\{B_1,\ldots,B_k\}\in\mathcal{P}(n)$. For $X_1,\ldots,X_n\in \mathcal{A}$ let
%		\begin{align*}
%		\varphi_\pi(X_1,\ldots,X_n)=\prod_{j=1}^k \varphi\left(X_{B_j}\right)
%		\end{align*}
%		where 
%		\begin{align*}
%		\varphi\left(X_{B_j}\right)=\varphi\left(\prod_{i\in B_j}X_i\right)
%		\end{align*}
%	\end{definition}
	Next we recall definitions of cumulant functionals related to non-crossing and interval partitions, called free and Boolean cumulants, respectively. Free cumulants,  introduced in \cite{SpeicherNC}, are important tools in free probability, while Boolean cumulants are related to the so called Boolean independence introduced in \cite{SpeicherWoroudi}.
	\begin{definition}
		For every $n\geq 1$ free cumulant functional $\kappa_n:\mathcal{A}^n\to \mathbb{C}$ is defined  recursively through equations
		\begin{align*}
		\forall m\geq 1\quad \forall (X_1,\ldots,X_m)\in \mathcal{A}^m\qquad\qquad \varphi(X_{1}\cdots X_{m})=\sum_{\pi\in NC(m)}\,\kappa_\pi(X_{1},\ldots,X_{m}),
		\end{align*}
		where for $\pi=\{B_1,\ldots,B_k\}\in NC(m)$
		\begin{align*}
		\kappa_\pi(X_1,\ldots,X_m)=\prod_{j=1}^k\,\kappa_{|B_j|}\left(X_i;i\in B_j\right).
		\end{align*}
		
		Similarly, for every $n\geq 1$  Boolean cumulant functional $\beta_n:\mathcal{A}^n\to \mathbb{C}$ is defined recursively through equations
		\begin{align*}
		\forall m\geq 1\quad \forall (X_1,\ldots,X_m)\in \mathcal{A}^m\qquad\qquad \varphi(X_1\cdots X_m)=\sum_{\pi\in Int(m)}\,\beta_\pi(X_1,\ldots,X_m),
		\end{align*}
		where for $\pi=\{B_1,\ldots,B_k\}\in Int(m)$ 
		\begin{align*}
			\beta_\pi(X_1,\ldots,X_m)=\prod_{j=1}^k \,\beta_{|B_j|}\left(X_i;i\in B_j\right).
			\end{align*}
	\end{definition}
	It turns out that freeness can be described in terms of free cumulants. More precisely random variables $X_1,\ldots,X_n$ are free if and only if for any $r\geq 2$ we have $\kappa_r\left(X_{i_1},\ldots,X_{i_r}\right)=0$ for any non-constant choice of $i_1,\ldots,i_r\in\{1,\ldots,n\}$.
	\begin{remark}
		In the sequel we will need the formula for Boolean cumulants with products as entries (see e.g. \cite{FMNS2})
		Fix two integers $m,n$ such that $0<m+1<n$ and numbers 
		$1\leq i_1<i_2<\ldots<i_{m+1}=n$. Denote by $\sigma$ the 
		interval partition 
		$\{\{1,\ldots,i_1\},\{i_1+1,\ldots,i_2\},\ldots,\{i_{m}+1,\ldots,i_{m+1}\}\}$. We have
		\begin{align}\label{BoolProd}
		\beta_{m+1}(X_1\cdots X_{i_1},\ldots,X_{i_{m}+1}\cdots X_{i_{m+1}})
		=\sum_{\substack{\pi \in Int(n)\\ \pi \vee \sigma=1_n}}\beta_\pi(X_1,\ldots,X_n),
		\end{align}
		where $\vee$ is join of partitions.
	\end{remark}
	In the present paper we will also need generating functions related with cumulants and moments:
	\begin{enumerate}
		\item Moment transform is defined for $z\in\mathbb{C}\setminus\mathbb{R}$ as
		\[M_X(z)=\int_{\mathbb{R}}\frac{zx}{1-zx}d\mu_X(x),\]
		if $X$ is bounded then in a neighbourhood od zero one has \[M_X(z)=z\varphi(X)+z^2\varphi(X^2)+\ldots\]
		\item $\eta$--transform is defined for $z\in\mathbb{C}\setminus\mathbb{R}$ as
		\[\eta_X(z)=\frac{M_X(z)}{1+M_X(z)}\]
		and in a neighbourhood of zero for bounded $X$ one has \[\eta_X(z)=\beta_1(X)z+\beta_2(X)z^2+\ldots\]
		\item For a positive $X$, for $z$ in a neighbourhood of $(-1,0)$ one can define so called $S$--transform
		\[S_X(z)=\tfrac{1+z}{z} M_X^{-1}(z)\] 
		which for free $X,Y$ has a remarkable property that  $S_{XY}=S_XS_Y$.
	\end{enumerate}

\subsection{Conditional expectation}
	Assume that $(\mathcal{A},\varphi)$ is a  $W^*$-probability spaces, that is
	$\mathcal{A}$ is a finite von Neumann algebra and $\varphi$ a faithful normal
	tracial state. Then for any von Neumann subalgebra $\mathcal{B}\subset
	\mathcal{A}$  there exists a faithful, normal projection
	$\E_\mathcal{B}:\mathcal{A}\to \mathcal{B}$ such that
	$\varphi\circ\E_\mathcal{B}=\varphi$.
	This projection is called the conditional expectation onto the subalgebra $\mathcal{B}$ with respect to $\varphi$. If $X\in \mathcal{A}$ is self-adjoint then $\E_\mathcal{B}(X)$ defines a unique self-adjoint element in $\mathcal{B}$. For $X\in\mathcal{A}$ by $\E_X$ we denote the conditional expectation given the von Neumann subalgebra generated by $X$ and $1_\mathcal{A}$. 
	
	An important tool for dealing with conditional expectations is the following equivalence 
	
	\begin{align}\label{eqn:equiv}
	\E_\mathcal{B}(Y)=Z\quad \Leftrightarrow  \quad Z\in\mathcal{B}\quad\mbox{and}\quad\varphi(YX)=\varphi(ZX)\quad\forall\,X\in\mathcal{B}.
	\end{align}
\subsection{Free Poisson and free Binomial distributions}
		In this subsection we recall definitions and some basic facts two about distributions: free Poisson and free binomial, which play important role in this paper. 
		
\begin{remark}[Free Poisson distribution]\label{rem:freePoisson} $\ $
		\begin{enumerate}
			\item The Marchenko--Pastur (or free Poisson) distribution $\mu=\mu(\alpha, \lambda)$ is defined by 
			\begin{align*}%\label{MPdist}
			\mu=\max\{0,\,1-\lambda\}\,\delta_0+\tilde{\mu},
			\end{align*}
			where $\alpha,\lambda> 0$ and the measure $\tilde{\mu}$, supported on the interval $(\alpha(1-\sqrt{\lambda})^2,\,\alpha(1+\sqrt{\lambda})^2)$, has the density (with respect to the Lebesgue measure)
			$$
			\tilde{\mu}(d x)=\frac{1}{2\pi\alpha x}\,\sqrt{4\lambda\alpha^2-(x-\alpha(1+\lambda))^2}\,d x. 
			$$
			\item For free Poisson distribution $\mu(\alpha,\lambda)$ the $S$-transform is of the form
			\begin{align*}
				S_{\mu(\alpha,\lambda)}(z)=\frac{1}{\alpha\lambda+\alpha z}
			\end{align*}
		\end{enumerate}
	\end{remark}

	\begin{remark}[Free binomial distribution] $\ $
		
		\begin{enumerate}
		\item Free binomial distribution $\nu=\nu(\sigma,\theta)$ is defined by
		\be\label{freebeta}
		\nu=(1-\sigma)\mathbb{I}_{0<\sigma<1}\,\delta_0+\tilde{\nu}+(1-\theta)\mathbb{I}_{0<\theta<1}\delta_1,
		\ee
		where $\tilde{\nu}$ is supported on the interval $(x_-,\,x_+)$,
		\begin{align}
		\label{binom_supp}
		x_{\pm}=\left(\sqrt{\frac{\sigma}{\sigma+\theta}\,\left(1-\frac{1}{\sigma+\theta}\right)}\,\pm\,\sqrt{\frac{1}{\sigma+\theta}\left(1-\frac{\sigma}{\sigma+\theta}\right)}\right)^2,
		\end{align}
		and has the density
		$$
		\tilde{\nu}(dx)=(\sigma+\theta)\,\frac{\sqrt{(x-x_-)\,(x_+-x)}}{2\pi x(1-x)}\,dx.
		$$
		where
		$(\sigma,\theta)\in \left\{(\sigma,\theta):\,\frac{\sigma+\theta}{\sigma+\theta-1}>0,\,\frac{\sigma\theta}{\sigma+\theta-1}>0\right\}$.
		The n-th free convolution power of distribution
		$$
		p\delta_0+(1-p)\delta_{1/n}
		$$
		is free-binomial distribution with parameters $\sigma=n(1-p)$ and $\theta=np$, which justifies the name of the distribution (see \cite{SaitohYosida}).

		\item For free Binomial distribution $\nu(\sigma,\theta)$ the $S$-transform is of the form
		\begin{align*}
		S_{\nu(\sigma,\theta)}(z)=1+\frac{\theta}{\sigma+z}.
		\end{align*}
		\end{enumerate}		
	\end{remark}
	
\section{Calculation of conditional expectation}

This section is devoted to derive explicit expressions for a family of conditional expectations of functions of free variables, as announced 
in the introduction. To attain this goal we use Boolean cumulants and some of the results from \cite{LehnerSzpojan,FMNS2}. Surprisingly, though the variables involved are free, calculations based on Boolean cumulants are much simpler than those based on free cumulants and shown
in \cite{SzpDLRNeg}. 

We start with recalling relevant facts about subordination, for details see \cite{BianeProceFreeIncr}.

Let $(\mathcal{A},\varphi)$ be as in Section 2.2. For a variable $X\in\mathcal{A}$ and $z\in \mathbb{C}\setminus\mathbb{R}$ we define a function $\Psi$ by $\Psi_{X}(z):=z X\left(1-z X\right)^{-1}$ and the following subordination formulas for conditional expectations with respect to positive $V$ and $U$, respectively, hold 
\begin{align}\E_V\,\Psi_{V^{1/2}UV^{1/2}}(z)=\Psi_V(\omega_1(z))\label{w1}\\
\E_U\,\Psi_{U^{1/2}VU^{1/2}}(z)=\Psi_U(\omega_2(z)).\label{w2}
\end{align}

Here, $\omega_1$ and $\omega_2$ are the subordination functions. Note that since we assume that $\varphi$ is tracial and both $U$ and $V$ are positive the moments of $UV,U^{1/2}VU^{1/2}$ and $V^{1/2}UV^{1/2}$ are the same, so $M_{UV}=M_{U^{1/2}VU^{1/2}}=M_{V^{1/2}UV^{1/2}}.$ Since by the very definition we have $M_X(z)=\varphi(\Psi_X(z))$  identities \eqref{w1} and \eqref{w2} imply
	\begin{equation}
	\label{w12}
	M_{UV}(z)=M_V(\omega_1(z))=M_U(\omega_2(z)).
	\end{equation}
For future use we will also denote $\Psi_X:=\Psi_X(1)$. In the sequel, we will also use the symbol $\psi$ for a formal power series $\psi(x)=\sum_{k\ge 1}\,x^k$, where $x$ is from some (unspecified) algebra over a real vector space.  

We will also need the following two formulas involving  Boolean cumulants. The first is taken from \cite{LehnerSzpojan} and the second is a simple consequence of Theorem 1.2 from \cite{FMNS2}; it is also closely related to the characterization of freeness from \cite{JekelLiu:2019:operad}.
\begin{proposition}
	Let $\{X_1,\ldots,X_{n+1}\}$ and $\{Y_1,\ldots,Y_n\}$ be free, $n\ge 1$. Then 
	\begin{align}
	\label{bocu1}
	&\varphi(X_1Y_1\ldots X_nY_n)\\ \nonumber &=\sum_{k=0}^{n-1}\,\sum_{0=j_0<j_1<\ldots<j_{k+1}=n}\,\varphi(Y_{j_1}\ldots Y_{j_{k+1}})\,\prod_{\ell=0}^k\,\beta_{2(j_{\ell+1}-j_{\ell})-1}(X_{j_{\ell}+1},Y_{j_{\ell}+1}\ldots,Y_{j_{\ell+1}-1},X_{j_{\ell+1}})
	\end{align}
	and
	\begin{align}
	\label{bocu2}
	&\beta_{2n+1}(X_1,Y_1,\ldots,X_n,Y_n,X_{n+1})\\ \nonumber &=\sum_{k=2}^{n+1}\,\sum_{1=j_1<\ldots<j_k=n+1}\,\beta_k(X_{j_1},\ldots,X_{j_k})\,\prod_{\ell=1}^{k-1}\,\beta_{2(j_{\ell+1}-j_{\ell})-1}(Y_{j_l},X_{j_{\ell}+1},\ldots,X_{j_{\ell+1}-1},Y_{j_{\ell+1}-1}).
	\end{align}
\end{proposition}

We note also a simple consequence of \eqref{bocu1} together with its reformulation.

\begin{remark}\label{rem:21} $\ $
\begin{enumerate}
	\item  
	In some neighbourhood of zero for subordination functions defined by equations \eqref{w1} and \eqref{w2} we have  (see \cite{LehnerSzpojan})
	\begin{align*}
	\omega_1(z)=\sum_{k=1}^\infty\beta_{2k-1}(U,V,U,\ldots,V,U)z^k,
	\end{align*}
	\begin{align*}
	\omega_2(z)=\sum_{k=1}^\infty\beta_{2k-1}(V,U,V\ldots,U,V)z^k.
	\end{align*}

	\item Equation \eqref{bocu1} can be reformulated in a, seemingly less transparent, but quite useful way, where we set $i_0:=0$ and  $i_k$ records the distance between $k$-th and $(k+1)$-st of $Y$'s which were picked to the outer block together with $Y_n$,
	\begin{align} \label{eq:mom-boolcum}
		\varphi\left(X_1Y_1\ldots X_nY_n\right)=
		&\sum_{k=1}^{n}\sum_{\substack{ i_1
			+\ldots+ i_k=n-k\\ i_1,\ldots,i_k\geq 0}}
		\varphi\left(Y_{n-i_0-i_1-\ldots-i_{k-1}-(k-1)}\ldots Y_{n-i_0-i_1-1}Y_n\right)\\ \nonumber
		&\prod_{j=1}^{k}\beta_{2i_j+1}\left(X_{n-i_0-i_1-\ldots-i_j-(j-1)},Y_{n-i_0-i_1-\ldots-i_j-(j-1)},\ldots,X_{n-i_0-i_1-\ldots-i_{j-1}-(j-1)}\right),
	\end{align}
	with $i_0=0$.
\end{enumerate}
\end{remark}

The remaining part of this section is devoted to calculation of conditional expectation \eqref{dualLukacsCondExp}. Actually, we will calculate a more general conditional expectation
\begin{align}\label{fgce}
\E_V\,f(U)U^{-1/2}\Psi_{U^{1/2}VU^{1/2}}(z)U^{-1/2}g(U),
\end{align}
where $f, g$ are functions such that $f(U)$ and $g(U)$ bounded. Note that the conditional expectation \eqref{fgce} for $f=g=\psi$ reduces to \eqref{dualLukacsCondExp}.

Below, in Prop. \ref{prop:fgce} and Prop. \ref{TT}, we present main technical results of this paper, i.e. we derive an explicit form of the conditional expectation \eqref{fgce}. It is worth to point out that this explicit form is written quite naturally just in terms of moments and Boolean cumulants and with no reference to free cumulants. This fact strengthens considerably the methodological recommendation from \cite{LehnerSzpojan}:  to calculate conditional moments of functions of free random variables  use Boolean cumulants!
\begin{proposition}\label{prop:fgce}
	Let $(\mathcal{A},\varphi)$ be as in Section 2.2. Assume that $U,V\in\mathcal{A}$ are free, $0\leq U<1$ and $V$ is bounded. Let $f$ and $g$ be such that $f(U)$ and $g(U)$ are bounded. Then for $z$ in some neighbourhood of $0$ and $\omega_1,\omega_2$  satisfying \eqref{w1} and \eqref{w2}
	\begin{align}\label{ev}
	&\E_V\,f(U)U^{-1/2}\Psi_{U^{1/2}VU^{1/2}}(z) U^{-1/2}g(U) \\
	&\hspace{40mm}=\omega_2(z)\eta_U^{f,g}(\omega_2(z))+z \eta_U^f(\omega_2(z))\eta_U^g(\omega_2(z))V\left(1+\Psi_V(\omega_1(z)\right),\nonumber
	\end{align}
	where
	\begin{align}
	\eta^{f,g}_U(z)&=\sum_{\ell=0}^{\infty}\,\beta_{l+2}(f(U),U,\ldots,U,g(U))z^{\ell},\label{etaufg}\\
	\eta^f_U(z)&=\sum_{\ell=0}^{\infty}\,\beta_{l+1}(f(U),U,\ldots,U)z^{\ell}.\label{etauf}
	\end{align}
\end{proposition}

\begin{proof}
	We will calculate the conditional expectation \eqref{fgce} using Remark \ref{rem:21}.
	
	For $z$ sufficiently small we can write
	$\Psi_{U^{1/2}VU^{1/2}}(z)=\sum_{n=1}^{\infty}z^{n}U^{1/2}(V(UV)^{n-1})U^{1/2}$. It suffices to calculate
	\[\sum_{n=1}^{\infty}z^n\varphi\left(f(U)V (UV)^{n-1} g(U)H\right),\]
	for any $H$ in the von Neumann algebra generated by $\{1_{\mathcal{A}},V\}$. Thus we reduce the problem to calculation of moments
	\[\varphi\left(f(U)V (UV)^{n-1} g(U)\,H\right)\quad n\ge 1.\]
	We use \eqref{eq:mom-boolcum} to express the above moment in terms of moments of $H$ and $V$ and Boolean cumulants of $f(U),g(U),U$ and $V$. The two free families in \eqref{eq:mom-boolcum} are $\{\underbrace{f(U),U,\ldots,U,g(U)}_{n+1}\}$ and $\{\underbrace{V,\ldots,V,H}_{n+1}\}$. Denoting $i_{\ell}=j_{\ell+1}-j_{\ell}$ in  \eqref{eq:mom-boolcum} we get
	%		\small
	\begin{align*}
	&\varphi\left(f(U)V (UV)^{n-1}g(U)H\right)=
	\varphi(H)\beta_{2n+1}\left(f(U),V,U,\ldots,U,V,g(U)\right)\\
	&+\sum_{k=1}^{n}\varphi\left(V^{k}H\right)\,\sum_{ i_1
		+\ldots+i_{k+1}=n-k}\beta_{2i_1+1}(f(U),V,\ldots,V,U) \cdots\beta_{2i_{k+1}+1}(U,V,\ldots,U,V,g(U)),
	\end{align*}
	%		\normalsize
	where the Boolean cumulants which are hidden under $\cdots$ in the formula above are of the form $\beta_{2k+1}(U,V,\ldots,U,V,U)$.
	
	Thus, taking into account \eqref{eqn:equiv}, the conditional expectation assumes the form \small
	\begin{align}
	&\E_V\,f(U)U^{-1/2}\Psi_{U^{1/2}VU^{1/2}}(z) U^{-1/2}g(U)
	=\sum_{n=1}^{\infty}z^n\,\beta_{2n+1}\left(f(U),V,U,\ldots,U,V,g(U)\right)\label{rhs}\\&
	+\sum_{n=1}^{\infty}z^n\,\left(\sum_{k=1}^{n}V^{k}\sum_{ i_1
		+\ldots+i_{k+1}=n-k}\beta_{2i_1+1}(f(U),V,\ldots,V,U)\cdots\beta_{2i_{k+1}+1}(U,V,\ldots,U,V,g(U))\right).\nonumber
	\end{align}\normalsize
	Let us denote
	\begin{align*}
	A_1(z)&=\sum_{n=0}^{\infty}\beta_{2n+1}\left(f(U),V,U,\ldots,U,V,U\right)z^{n},\quad
	A_2(z)=\sum_{n=0}^{\infty}\beta_{2n+1}\left(U,V,U,\ldots,U,V,g(U)\right)z^{n},\\
	B(z)&=\sum_{n=1}^{\infty}\beta_{2n+1}\left(f(U),V,U,\ldots,U,V,g(U)\right)z^{n},\\
	C_1(z)&=\sum_{n=0}^{\infty}\beta_{2n+1}\left(U,V,U,\ldots,U,V,U\right)z^{n},\quad 
	C_2(z)=\sum_{n=0}^{\infty}\beta_{2n+1}\left(V,U,V\ldots,V,U,V\right)z^{n}.
	\end{align*}
	
	Observe that the change of order of summation in the second summand at the right hand side of \eqref{rhs} gives \small
	\begin{align*}
	& \sum_{n=1}^{\infty}z^n\sum_{k=1}^{n}V^{k}\sum_{ i_1
		+\ldots+i_{k+1}=n-k}\beta_{2i_1+1}(f(U),V,\ldots,V,U) \cdots\beta_{2i_{k+1}+1}(U,V,\ldots,U,V,g(U))\\
	&  =z V \sum_{n=1}^{\infty}\sum_{k=1}^{n}V^{k-1}z^{k-1} \sum_{ i_1
		+\ldots+i_{k+1}=n-k}\beta_{2i_1+1}(f(U),V,\ldots,V,U)z^{i_1}\cdots\beta_{2i_{k+1}+1}(U,V,\ldots,U,V,g(U))z^{i_{k+1}}\\
	& =zA_1(z)A_2(z)V\sum_{k=1}^{\infty}\,[zC_1(z)V]^{k-1}=zA_1(z)A_2(z)V\left(1+\Psi_V(z C_1(z))\right).
	\end{align*}\normalsize
	Remark \ref{rem:21} implies that $zC_1(z)$ is exactly the subordination function $\omega_1(z)$.
	
	Thus,  returning to \eqref{rhs} and the definition of $B$ we finally get
	\begin{align*}
	\E_V\,f(U)U^{-1/2}\Psi_{U^{1/2}VU^{1/2}}(z) U^{-1/2}g(U)=B(z)+z A_1(z)A_2(z)V(1+\Psi_V(\omega_1(z)).
	\end{align*}		

	We will use below \eqref{bocu2} for free collections $\{\underbrace{f(U),U,\ldots,U}_{i+1}\}$ and $\{\underbrace{V,\ldots,V}_i\}$ for any $i\ge 0$ (with indices changed in the way we did when using \eqref{eq:mom-boolcum} earlier in this proof)	
	\begin{align*}
	&A_1(z)=\sum_{i=0}^{\infty}z^i\beta_{2i+1}(f(U),V,\ldots,V,U)\\
	&=\sum_{i=0}^\infty\sum_{l=0}^i\beta_{l+1}(f(U),U,\ldots,U)z^{l}\sum_{i_1+\ldots+i_l=i-l}\beta_{2i_1+1}(V,U,\ldots,U,V)z^{i_1}\ldots\beta_{2i_l+1}(V,U,\ldots,U,V)z^{i_l}\\
	&=\sum_{l=0}^{\infty}\beta_{l+1}(f(U),U,\ldots,U)z^{l}C_2^l(z)
	\end{align*}
	where $zC_2(z)=\omega_2(z)$. Consequently, $A_1(z)=\eta^f_U(\omega_2(z))$ - see \eqref{etauf}. 
	
	Boolean cumulants are invariant under reflection $$\beta_{2i+1}(U,V,\ldots,U,V,g(U))=\beta_{2i+1}(g(U),V,U,\ldots,V,U)$$ and thus, by the same calculation as for $A_1$, we get $A_2(z)=\eta^g_U(\omega_2(z))$.

	To find $B$ we repeat (with obvious modifications) the first part of the calculation which has been done above for $A_1$. It gives the representation
	\begin{align*}
	B(z)=\sum_{l=0}^{\infty}\beta_{l+2}(f(U),U,\ldots,U,g(U))z^{l+1}C_2^{l+1}(z).
	\end{align*}
	Consequently, 	$B(z)=\omega_2(z)\eta_U^{f,g}(\omega_2(z))$ - see \eqref{etaufg}.

\end{proof}

It turns out that both $\eta^{f,g}_U$ and $\eta_U^f$ can be conveniently expressed in terms of an operation $\varphi_D^U$ which we are going to define now.
 
Let $H$ be  a formal power series $H(T)=\sum_{k\ge 0} h_k T^k$ where $T$ is a variable from some algebra. Let ${\bf D}$ denote the zero derivative i.e. a linear operator which is defined on a  power series in $T\in \mathcal{A}$ through its action on monomials: ${\bf D}\,T^k=T^{k-1}$ for $k\ge 1$ and ${\bf D}\,T^k=0$ for $k=0$. Similarly by $D$ we denote the zero derivative acting  on complex  power series.
  
For a power series $f:\mathcal{A}\to\mathcal{A}$ and $T\in\mathcal{A}$ we define a new operator $\varphi_D^T(H,f)$  (acting on complex functions)
\begin{equation}\label{varp}
\varphi_D^T(H,f):=\sum_{k\ge 0}\,h_k\varphi\left({\bf D}^kf(T)\right)D^k.
\end{equation}
Note that $\varphi_D^T(H,f)$ is a formal series  of weighted zero derivatives of increasing orders (with the weight $h_k\varphi\left({\bf D}^kf(T)\right)$ for the derivative of order $k$, $k\ge 0$). The result of its application to an analytic function in general is a formal series (which in some cases may converge).  Note that for integer $r\ge 0$
$$
\varphi_D^T(H,\,T^r)=\sum_{k=0}^r\,h_k\varphi(T^{r-k})\,D^k,
$$
and, in particular, $$\varphi_D^T(H,\,1)=h_0\qquad\mbox{and}\qquad\varphi_D^T(H,\,T)=h_0\varphi(T)\mathrm{id}+h_1D.$$
Therefore, e.g. for $H:=\psi$ we have
\begin{equation}
\label{DDD}
\varphi_D^T(\psi,\,1)=0\qquad \mbox{and}\qquad \varphi_D^T(\psi,\,T)=D.
\end{equation}

On the other hand, since for $\psi(z)=z(1-z)^{-1}$, $|z|<1$, we have $D^k\psi=1+\psi$, it follows that
$$
\varphi_D^T(H,\psi)=\varphi(1+\Psi_T)H(D).
$$
In particular,
\begin{equation}\label{psipsi}
\varphi_D^T(\psi,\psi)=\varphi(1+\Psi_T)\psi(D).
\end{equation}

Now we are ready to give explicit formulas for $\eta_U^{f,g}$ and $\eta_U^f$ defined in \eqref{etaufg} and \eqref{etauf}, respectively.
\begin{proposition}\label{TT}
	Assume that $f$ and $g$ are  analytic functions on the unit disc. Let $U$ be a non-commutative variable  such that $0\leq U<1$ with the $\eta$--transform $\eta_U$. 	
	Then  
	\begin{align}\label{etafg}
	\eta^{f,g}_U=\left[\varphi_D^U(\psi,f)\circ\varphi_D^U(\psi,g)\right]\,\eta_U,	
	\end{align}
	and 
	\begin{align}\label{etaf}
	\eta_U^f(z)=z\,([\varphi_D^U(\psi,\,f)\, \circ D]\eta_U)(z)+(\varphi_D^U(\psi,\,f)\,\eta_U)(0).	
	\end{align}
\end{proposition}

\begin{proof} First, we consider $\eta^{f,g}_U$.
	
	Let $f(z)=\sum_{k\ge 0}\,a_kz^k$ and $g(z)=\sum_{k\ge 0}\,b_kz^k$. Expanding $f$ and $g$ we get 
	\begin{align*}
	\eta^{f,g}_U(z)&=\sum_{\ell=0}^{\infty}\,z^{\ell}\sum_{i,j\ge 0}a_ib_j\beta_{l+2}(U^i,U,\ldots,U,U^j).
	\end{align*}
	Since, $\beta_{l+2}(1,U,\ldots,U,U^j)=\beta_{l+2}(U^i,U,\ldots,U,1)=0$ the inner double sum starts with $i=j=1$. 
	
	Using the definition of Boolean cumulants and the formula for Boolean cumulants with products as entries \eqref{BoolProd} one can easily obtain the following formula
	%\jacek{wymaga dowodu} \kamil{Dodalem wzor na kumulanty Boolowskie z iloczynami, dowod z wykorzystaniem tego wzoru jest natychmiastowy, uwazam ze tyle wystarczy}
	\begin{align}\label{eq:booleanPowers}
		\beta_{r+1}\left(G,U,\ldots,U,U^i\right)=\beta_{r+1}\left(U^i,U,\ldots,U,G\right)=\sum_{m=1}^i\beta_{r+m}(G,\underbrace{U,\ldots,U}_{r+m-1})\varphi\left(U^{i-m}\right).
	\end{align}

Applying \eqref{eq:booleanPowers} we obtain
$$
\eta^{f,g}_U(z)=\sum_{\ell=0}^{\infty}\,z^{\ell}\sum_{i,j\ge 1}a_ib_j\sum_{k=1}^i\,\sum_{m=1}^j\beta_{\ell+k+m}(U)\varphi(U^{i-k})\varphi(U^{j-m}).
$$
Changing several times the order of inner summations and the variables we get
\begin{align*}
\eta^{f,g}_U(z)&=\sum_{\ell=0}^{\infty}\,z^{\ell}\,\sum_{k,m\ge 1}\,\beta_{l+k+m}(U)\,\varphi\left(\sum_{i\ge k}\,a_iU^{i-k}\right)\,\varphi\left(\sum_{j\ge m}\,b_jU^{j-m}\right)\\
&=\sum_{\ell=0}^{\infty}\,z^{\ell}\,\sum_{k,m\ge 1}\,\beta_{l+k+m}(U)\,\varphi\left({\bf D}^kf(U)\right)\,\varphi\left({\bf D}^mg(U)\right)\\
& =\sum_{\ell=0}^{\infty}\,z^{\ell}\,\sum_{r=2}^{\infty}\,\beta_{\ell+r}(U)\,\sum_{k=1}^{r-1}\,\varphi\left({\bf D}^kf(U)\right)\,\varphi\left({\bf D}^{r-k}g(U)\right)\\
&= \sum_{r=2}^{\infty}\,\left(\sum_{k=1}^{r-1}\,\varphi\left({\bf D}^kf(U)\right)\,\varphi\left({\bf D}^{r-k}g(U)\right)\right)\,\sum_{\ell=0}^{\infty}\,\beta_{\ell+r}(U)z^{\ell}\\
&=\sum_{r=2}^{\infty}\,\left(\sum_{k=1}^{r-1}\,\varphi\left({\bf D}^kf(U)\right)\,\varphi\left({\bf D}^{r-k}g(U)\right)\right)\,D^r\eta_U(z).
\end{align*}
Consequently, 
\begin{align*}
\eta^{f,g}_U&=\sum_{r=2}^{\infty}\,\left(\sum_{k=1}^{r-1}\,\left[\varphi({\bf D}^kf(U))D^k\right]\,\left[\varphi({\bf D}^{r-k}g(U)) D^{r-k}\right]\right)\,\eta_U\\
&=\left[\sum_{k\ge 1}\,\varphi\left({\bf D}^kf(U)\right) D^k\right]\circ\left[\sum_{k\ge 1}\,\varphi({\bf D}^kg(U)) D^k\right]\,\eta_U.
\end{align*}

Second, we consider $\eta_U^f$. By definition  \begin{equation}
\label{fgf}
\eta_U^f(z)=z\eta_U^{f,\mathrm{id}}(z)+\beta_1(f(U)).\end{equation} 

Note that \eqref{etafg} for $g=\mathrm{id}$ together with the second identity of \eqref{DDD} yields
\begin{equation}\label{fid}
\eta_U^{f,\mathrm{id}}=[\varphi^D_U(\psi,f)\circ D]\eta_U.
\end{equation}

For $f(z)=\sum_{i\ge 1}a_iz^i$ we have
	$$
	\beta_1(f(U))=\sum_{i\ge 1}a_i\beta_1(U^i)
	$$
	Thus \eqref{eq:booleanPowers}  yields
	$$
	\beta_1(f(U))=\sum_{i\ge 1}a_i\sum_{k=1}^i\beta_k(U)\varphi(U^{i-k})
	=\sum_{k\ge 1}\beta_k(U)\varphi({\bf D}^kf(U)).
	$$
	Since $\beta_k(U)=(D^k\eta_U)(0)$ we see that
\begin{equation}
\label{beta1}
	\beta_1(f(U))=\sum_{k\ge 1}\,\varphi({\bf D}^kf(U))(D^k\eta_U)(0)=(\varphi_D^U(\psi,\,f)\,\eta_U)(0).
	\end{equation}
	
The final result follows now by inserting \eqref{fid} and \eqref{beta1} into \eqref{fgf}.
\end{proof}
\begin{remark}
	Note that due to the second formula in \eqref{DDD} it follows from \eqref{etafg} that
	$$\eta^{\mathrm{id},\mathrm{id}}_U=D^2\eta_U$$
	and \eqref{etaf} yields
	$$
	\eta^{\mathrm{id}}_U=D\eta_U.
	$$
	The last identity extends to any function $f_r$ defined as $f_r(T)=T^r$, where $r\ge 1$ is an integer, as follows
	$$
	\eta_U^{f_r}=\sum_{j=1}^r\,\varphi(U^{r-j})D^j\eta_U.
	$$

	\end{remark}
\begin{remark}
	One can rewrite the equation \eqref{etafg} in a more straightforward form as
	\begin{align*}
	\eta^{f,g}_U(z)=\sum_{r=2}^{\infty}\,\left(\sum_{k=1}^{r-1}\,\varphi\left(\sum_{i\ge k}\,a_iU^{i-k}\right)\,\varphi\left(\sum_{j\ge r-k}\,b_jU^{j-(r-k)}\right)\right)\,\sum_{\ell=0}^{\infty}\,\beta_{\ell+r}(U)z^{\ell}.
	\end{align*}
	Similarly \eqref{etaf} expands into
	\begin{align*}
	\eta^f_U(z)=\sum_{r=1}^{\infty}\,\varphi\left(\sum_{i\ge r}\,a_iU^{i-r}\right)\,\sum_{\ell=0}^{\infty}\,\beta_{\ell+r}(U)z^{\ell}.
	\end{align*}
	
\end{remark}

Finally, we calculate the conditional expectation \eqref{dualLukacsCondExp} which will be needed in Section 4  in the proof of the second one of the regression characterizations.
\begin{proposition}\label{lem:condExp}
	Let $(\mathcal{A},\varphi)$ be as in Section 2.2.
	Assume that $U,V\in\mathcal{A}$ are free, $0\leq U<1$ and $V$ is bounded. Then for $z$ in some neighbourhood of $0$ 
	\begin{align}\label{evu}
	\E_V\,(1-U)^{-1}U^{1/2}\Psi_{U^{1/2}VU^{1/2}}(z) U^{1/2}(1-U)^{-1}=B(z)+z A^2(z)V\left(1+\Psi_V(\omega_1(z)\right),
	\end{align}
	where
	\begin{align}
	A(z)=&\frac{\eta_U(\omega_2(z))-\eta_U(1)}{\omega_2(z)-1}\,\varphi\left((1-U)^{-1}\right),\label{az}\\
	B(z)=&\frac{\omega_2(z)(\eta_U(\omega_2(z))-\eta_U(1)-(\omega_2(z)-1)\eta_U'(1))}{(\omega_2(z)-1)^2}\,\varphi^2\left((1-U)^{-1}\right)\label{bz}
	\end{align}
	and $\omega_1,\omega_2$  satisfy \eqref{w12}.
\end{proposition}
\begin{proof}
	From \eqref{ev} with $f=g=\psi$ we see that the conditional expectation at the left hand side of \eqref{evu} is of the form
	$$
	\omega_2(z)\eta_U^{\psi,\psi}(\omega_2(z))+z \left(\eta_U^{\psi}(\omega_2(z))\right)^2\left(1+\Psi_V(\omega_1(z))\right)V,
	$$
	i.e. we need only to show that \begin{itemize} \item $\eta_U^{\psi}(\omega_2(z))=A(z)$, where $A$  is defined in \eqref{az}; \end{itemize}
		
		 and 
		 \begin{itemize}
		 \item $\omega_2(z)\eta_U^{\psi,\psi}(\omega_2(z))=B(z)$, where $B$ is defined in \eqref{bz}. 
		 \end{itemize}
	 
	We first compute $\eta_U^{\psi}$. To this end we rely on \eqref{etaf}. Observe that \eqref{psipsi} yields 
	\begin{align*}
	\varphi_D^U(\psi,\psi)(D\eta_U)=\varphi(1+\Psi_U)\sum_{k\ge 1}D^{k+1}\eta_U=\varphi(1+\Psi_U)(D\psi(D))\eta_U.
	\end{align*}
	Note that for a power series $h(z)=\sum_{j\ge 0}\,h_jz^j$ we have
	\begin{align}
	\label{DD1}
	\psi(D)\,(h)(z):=\sum_{k\ge 1}\,D^kh(z)=\sum_{k\ge 1}\sum_{j\ge k}h_jz^{j-k}=\sum_{j\ge 1}\,h_j\sum_{k=1}^j\,z^{j-k}=\tfrac{h(z)-h(1)}{z-1}.
	\end{align}
	Therefore, using \eqref{DD1} we get
	$$
	D\psi(D)\,(h)(z)=D\left(\tfrac{h(z)-h(1)}{z-1}\right)=\tfrac{Dh(z)-h(1)}{z-1}.
	$$
	Thus, since $zD\eta_U(z)=\eta_U(z)$, we get
	$$
	z\varphi_D^U(\psi,\,\psi)(D\eta_U)(z)=\varphi(1+\Psi_U)\tfrac{\eta_U(z)-z\eta_U(1)}{z-1}.
	$$
	
	Note also that due \eqref{DD1}  and  $\eta_U(0)=0$ we obtain 
	$$
	\varphi_D^U(\psi,\,\psi)(\eta_U)(0)=\varphi(1+\Psi_U)(\psi(D)\eta_U)(0)=\varphi(1+\Psi_U)\eta_U(1)
	$$
	Thus, \eqref{etaf} implies $A(z):=\eta_U^{\psi}(\omega_2(z))$ .

	Now we calculate $\eta_U^{\psi,\psi}$. Using  \eqref{etafg} and then first only once referring to \eqref{DD1}, we get
	$$
	\eta_U^{\psi,\psi}=\varphi^2(1+\Psi_U)\,\psi^{\circ 2}(D)\eta_U=\varphi^2(1+\Psi_U)\,\psi(D)\,K,$$
	where the function $K$ is defined by 
	$$
K(w)=\tfrac{\eta_U(w)-\eta_U(1)}{w-1}.
	$$
	Since $K(1)=\eta_U'(1)$, upon using again \eqref{DD1},  we  get
	$$
	\eta_U^{\psi,\psi}(w)=\varphi^2(1+\Psi_U)\,\tfrac{\tfrac{\eta_U(w)-\eta_U(1)}{w-1}-\eta_U'(1)}{w-1}
	$$
	and the formula for $B$ follows.
\end{proof}
\section{Dual Lukacs regressions of negative orders}

In this section we present simplified proofs of results from \cite{SzpDLRNeg}. We start with the regression 
characterization which can be approached just  by subordination 
technique, see \cite{EjsmontFranzSzpojankowski:2015:convolution}, without necessity to refer to Boolean 
cumulants.
\begin{theorem} \label{thm:11}
	Let $(\mathcal{A},\varphi)$ be a $W^*$ probability space. Let $U,V\in\mathcal{A}$ be free and such that $0<U<1$ and $V>0$. Assume that for some $b,c\in\R$
	\begin{align}
		\label{eq:1}
		\E_{V^{1/2}UV^{1/2}}\,V^{1/2}(1-U)V^{1/2}&=b \mathit{I},\\
	\label{eq:2}
			\E_{V^{1/2}UV^{1/2}}\,[V^{1/2}(1-U)V^{1/2}]^{-1}&=c \mathit{I}.
	\end{align}
	Then $b,c>0$, $bc>1$ and, with $\alpha=\varphi(\Psi_U)>0$, \begin{itemize}
		\item $V$ has free Poisson distribution $\mu\left(\tfrac{bc-1}{c},\,\tfrac{bc+\alpha}{bc-1}\right)$,
	\item  $U$ has free binomial distribution $\nu\left(\tfrac{\alpha}{bc-1},\,\tfrac{bc}{bc-1}\right)$.
	\end{itemize}
\end{theorem}

\begin{remark}
	This result holds, with the same proof, for unbounded random variables $U,V$ affiliated with respective von Neumann algebras, similarly as in \cite{EjsmontFranzSzpojankowski:2015:convolution}. The same framework was considered in \cite{BianeProceFreeIncr}. Then it is enough to assume that $U,V>0$ are such that $(1-U)^{-1}$ exists and
	$\varphi(V)$, $\varphi(V^{-1})$, $\varphi(U)$ and $\varphi((1-U)^{-1})$ are finite.
\end{remark}

\begin{proof}
	Consider first  \eqref{eq:2} which after multiplication by $\Psi_{V^{1/2}UV^{1/2}}(z)$ implies
	\begin{align}\label{uvu}
		L:=\varphi\left((1-U)^{-1}V^{-1/2}\Psi_{V^{1/2}UV^{1/2}}(z)V^{-1/2}\right)=c M_{V^{1/2}UV^{1/2}}(z).
	\end{align}
	It is easy to see through purely algebraic manipulations that for non-commutative $W>0$ and $T>0$
	\begin{align}\label{eq:Psi}
		W^{-1/2}\Psi_{W^{1/2}TW^{1/2}}(z)W^{-1/2}=z T^{1/2}\left(\Psi_{T^{1/2}WT^{1/2}}(z)+1\right) T^{1/2}.
	\end{align}
Therefore applying \eqref{eq:Psi} with $(W,T)=(V,U)$ to \eqref{uvu} we get 
	\begin{align*}
		L&=\varphi\left(z (1-U)^{-1}U^{1/2} (\Psi_{U^{1/2}VU^{1/2}}(z)+1)U^{1/2}\right)\\&=\varphi\left(z (1-U)^{-1}U^{1/2} (\E_U\,\Psi_{U^{1/2}VU^{1/2}}(z)+1)U^{1/2}\right)\\&=z\varphi\left( (1-U)^{-1}U(\Psi_U(\omega_2(z))+1)\right).
	\end{align*}
By simple algebra 
\begin{align}\label{simal}
U(1-U)^{-1}(\Psi_U(t)+1)=\tfrac{1}{t-1}\left(\Psi_U(t)-\Psi_U(1)\right)
\end{align}
and thus (we write below $\omega_2=\omega_2(z)$)
\begin{align*}
L=z\frac{M_U(\omega_2)-\alpha}{\omega_2-1}.
\end{align*}
	Consequently, \eqref{w12} and \eqref{uvu} yield
\begin{equation}\label{eq12}
z\left(M_U(\omega_2)-\alpha\right)=c(\omega_2-1)M_U(\omega_2).
\end{equation}

Similarly, we multiply both sides of \eqref{eq:1} by $z\Psi_{V^{1/2}UV^{1/2}}(z)$ and obtain
\begin{align}\label{em}
M:=z\varphi\left((1-U)V^{1/2}\Psi_{V^{1/2}UV^{1/2}}(z)V^{1/2}\right)=zb M_{V^{1/2}UV^{1/2}}(z).
\end{align}
Applying \eqref{eq:Psi} with $(W,T)=(U,V)$ we see that 
\[
M=\phi((1-U)U^{-1/2}\Psi_{U^{1/2}VU^{1/2}}(z)U^{-1/2})-z\phi((1-U)V).
\]
Thus traciality and subordination \eqref{w2} yield
	$$
M=\phi(\Psi_U^{-1}\,\Psi_U(\omega_2(z)))-zb
$$
since \eqref{eq:1} implies $\phi((1-U)V)=b$. We rewrite \eqref{simal} as
$$
\Psi_U^{-1}\Psi_U(t)=t+(t-1)\Psi_U(t)
$$
and plug it into $M$. Taking additionally into account \eqref{w12} at the RHS of \eqref{em} we finally get
\begin{equation}
\label{eq11}
\omega_2(z)+(\omega_2-1)M_U(\omega_2)=bz (M_U(\omega_2)+1).
\end{equation}

Identity $M_{UV}(z)=M_U(\omega_2(z))$ can be written in terms of inverse functions as $\omega_2(M_{UV}^{\<-1\>}(s))=M_U^{\<-1\>}(s)$ (for the discussion about the existence of an inverse see \cite{BelinschiBercoviciPartiallyDefSemigroups}). Thus rewriting \eqref{eq11} and \eqref{eq12} in terms of $M_{UV}^{\<-1\>}(s)$ and $M_U^{\<-1\>}(s)$ we obtain the following system of linear equations
\begin{equation}\label{linsys1}
\left\{\begin{array}{l}
b(1+s)M_{UV}^{\<-1\>}(s)=(1+s)M_U^{\<-1\>}(s)-s,\\
(s-\alpha)M_{UV}^{\<-1\>}(s)=cs\left(M_U^{\<-1\>}(s)-1\right).
\end{array}\right.
\end{equation}

We solve this system in terms of $M_U^{-1}$ and $M_{UV}^{-1}$ and thus  obtain $S$-transforms
\begin{align*}
S_U(s)= 1+\frac{bc}{\alpha+(bc-1)s}\\
S_{UV}(s)=\frac{c}{\alpha+s(bc-1)}.
\end{align*}
By freeness of $U$ and $V$ we know that $S_{UV}=S_U\,S_V$, which allows to compute the $S$-transform of $V$,
\begin{align*}
S_V(s)=\frac{c}{bc+\alpha+(bc-1)s}.
\end{align*}
Since the $S$-transform determines the distribution uniquely, the result follows.
\end{proof}

The next regression characterization cannot be 
proved just by referring to subordination as in the proof above. The 
proof we give below shows how useful can be enhancement of the subordination 
methodology with explicit formulas for conditional expectations 
expressed in terms of Boolean cumulants. We expect that this kind of 
approach can be of use also in other regression characterizations for 
which a simple subordination technique is not a sufficient tool, e.g. in regression characterizations related to
recent results for the free GIG and free Poisson see \cite{SzpojanMY} and of the free Kummer and free Poisson given in \cite{Piliszek}. \begin{theorem}
	Let $0<U<1$ and $V>0$ be free, $V$ bounded. Assume that for some  $c,d\in\R$ condition  \eqref{eq:2} holds and
	\begin{align}
		\E_{V^{1/2}UV^{1/2}}\,[V^{1/2}(1-U)V^{1/2}]^{-2}=d \mathit{I}.
		\label{eq:22}		
			\end{align}
Then $d>c^2$ and with $\alpha=\varphi(\Psi_U)$ \begin{itemize}
	\item $V$ has free Poisson distribution $\mu\left(\tfrac{d-c^2}{c^3},\,\tfrac{c^2\alpha+d}{d-c^2}\right)$,
	\item $U$ has free binomial distribution $\nu\left(\tfrac{c^2\alpha}{d-c^2},\,\tfrac{d}{d-c^2}\right)$.
\end{itemize} 
\end{theorem}
\begin{proof}
	We first observe that, as in the previous proof, \eqref{eq:2} implies \eqref{eq12}.
	
	Then we consider \eqref{eq:22}. We multiply its both sides by $\Psi_{V^{1/2}UV^{1/2}}(z)$ and apply state to get
	\begin{align}\label{en}
		N:=\varphi\left(V^{-1/2}(1-U)^{-1}V^{-1}(1-U)^{-1}V^{-1/2}\Psi_{V^{1/2}UV^{1/2}}\right)=d M_{V^{1/2}UV^{1/2}}
	\end{align}
	Using traciality of $\varphi$ and identity \eqref{eq:Psi} with $(W,T)=(V,U)$ we obtain
	\begin{align*}
		N&= z\varphi\left((1-U)^{-1}V^{-1}(1-U)^{-1}U^{1/2}\left(\Psi_{U^{1/2}VU^{1/2}}(z)+1\right) U^{1/2}\right)\\
		&=z\varphi\left(V^{-1}(1-U)^{-1}U^{1/2}\Psi_{U^{1/2}VU^{1/2}}(z) U^{1/2}(1-U)^{-1}\right)+z\varphi\left(V^{-1}(1-U)^{-2}U\right)\\
		&=z\varphi\left(V^{-1}\E_V\,(1-U)^{-1}U^{1/2}\Psi_{U^{1/2}VU^{1/2}}(z) U^{1/2}(1-U)^{-1}\right)+z\varphi\left(V^{-1}(1-U)^{-2}U\right).
	\end{align*}
	By Proposition \ref{lem:condExp} we get
	$$
	N=zB(z)\phi(V^{-1})+z^2A^2(z)(1+M_V(\omega_1(z)))+z\varphi\left(V^{-1}(1-U)^{-2}U\right).
	$$

	Since $\eta_U=\tfrac{M_U}{1+M_U}$ it follows that $\eta_U' =\tfrac{M_U'}{(1+M_U)^2}$. Also $M_U(1)=\phi\left(\sum_{k=1}^{\infty}\,U^k\right)=\varphi(\Psi_U)$ and $M_U'(1)=\phi\left(U\sum_{k=1}^{\infty}\,kU^{k-1}\right)=\phi(U(1-U)^{-2})$. Moreover, \eqref{eq:2} implies
	$$\varphi(V^{-1})\varphi((1-U)^{-1})=c\quad \mbox{and}\quad \phi(\Psi_U)=c\varphi(U)\phi(V)$$
	and \eqref{eq:22} yields
	$$
	\varphi(U(1-U)^{-2})\varphi(V^{-1})=d\varphi(U)\varphi(V).
	$$
	Additionally, we easily see that
	$$
	\varphi(V^{-1})=\frac{c}{\alpha+1},\quad \eta_U(1)=\frac{\alpha}{1+\alpha},$$$$ M_U'(1)=\phi(U(1-U)^{-2})=\tfrac{\alpha(1+\alpha)d}{c^2},\quad \eta_U'(1)=\frac{\alpha d}{(1+\alpha)c^2}.
	$$
Moreover, $M_V(\omega_1(z))=M_U(\omega_2(z))$. Summing up, (below $\omega_2=\omega_2(z)$)
	$$
	N=\tfrac{cz\omega_2}{(\omega_2-1)^2}\,\tfrac{M_U(\omega_2)-\alpha}{M_U(\omega_2)+1}+\left(\tfrac{z}{\omega_2-1}\right)^2\tfrac{(M_U(\omega_2)-\alpha)^2}{M_U(\omega_2)+1}-\tfrac{\alpha d z}{c(\omega_2-1)}.
	$$
Since $M_{V^{1/2}UV^{1/2}}(z)=M_U(\omega_2)$ at the RHS of \eqref{en} this equation upon multiplication both sides by $z$ can be written as
$$
\left(\tfrac{z}{c(\omega_2-1)}\right)^2\,\tfrac{M_U(\omega_2)-\alpha}{M_U(\omega_2+1}\,c^2\,\left(c\omega_2-z\left[M_U(\omega_2)-\alpha\right]\right)=dz\left(\tfrac{z}{c(\omega_2-1)}\,\alpha+M_U(\omega_2)\right).
$$
Using now \eqref{eq12} in the form $\tfrac{z}{c(\omega_2-1)}=\tfrac{M_U(\omega_2)}{M_U(\omega_2)+1}$  we get 
$$
\tfrac{c^2}{M_U(\omega_2)+1}\left(c\omega_2-z\left[M_U(\omega_2)-\alpha\right]\right)=dz.
$$

Now, plug in $z\left[M_U(\omega_2)-\alpha\right]=c(\omega_2-1)M_U(\omega_2)$, which is another reformulation of \eqref{eq12}, to conclude that
\begin{equation}\label{ome}
\omega_2+(\omega_2-1)M_U(\omega_2)=\tfrac{d}{c^3}z(M_U(\omega_2)+1).
\end{equation}
Note that \eqref{ome} upon substitution $b=d/c^3$ is the same as \eqref{eq11}. Therefore \eqref{eq12} and \eqref{ome} is the same system of equations as in the previous proof and thus the result follows.
\end{proof}
\section{Direct free Lukacs property for Marchenko-Pastur law}
Recall that in classical probability the following two implications are trivially
equivalent: 
\begin{itemize} \item if $X$ and $Y$ are independent gamma variables with the same scale parameter $\lambda$ and shape parameters $\alpha_X$ and $\alpha_Y$, respectively, then $U=X/(X+Y)$ and $V=X+Y$ are independent, $U$ is beta with parameters $\alpha_X,\,\alpha_Y$ and $V$ is gamma with the shape $\alpha_X+\alpha_Y$ and scale $\lambda$;
	\item if $U$ and $V$ are independent random variables, $U$ is beta with parameters $\alpha_X,\,\alpha_Y$ and $V$ is gamma with the shape $\alpha_X+\alpha_Y$ and scale $\lambda$, then $X=UV$ and $Y=(1-U)V$ are independent gamma variables with the same scale parameter $\lambda$ and shape parameters $\alpha_X$ and $\alpha_Y$, respectively.
\end{itemize}
However, the situation in free probability changes drastically, see the proofs of the dual Lukacs property in \cite{SzpojanWesol} and of the free direct Lukacs property in \cite{SzpLukacsProp}.
	
	  Until this moment we considered regression versions of the free dual Lukacs property which says that for $U$ and $V$ which are free and, respectively, free binomial and free Poisson distributions with properly interrelated parameters then $X=V^{1/2}UV^{1/2}$ and $Y=V^{1/2}(1-U)V^{1/2}$ are also free and have both free Poisson distributions. It appears that this free dual Lukacs property can be used as the main tool to construct a pocket proof of the direct free Lukacs property.  Let us note that the original proof of the direct Lukacs property for free Poisson distribution, as given in \cite{SzpLukacsProp}, relied heavily on combinatorics of free cumulants and non-crossing partitions. In particular, one of its highlights was explicit expression for the joint free cumulants of $X$ and $X^{-1}$ for free Poisson distributed $X$. 

Finally, before we present  the free direct Lukacs property and its proof, let us emphasize another methodological distinction between classical and free probability which is related to characterizations by independence/freeness.  In classical probability derivation of independence properties  (e.g. Lukacs or Matsumoto-Yor or Hamza-Vallois properties) is completely elementary (it relies on just derivation  of the jacobian of the considered transformation). Only the converse problems, characterizations, typically present serious mathematical challenges. For  free counterparts, as a rule, both the property and the characterization are challenging questions - see e.g. for the free version of the Matsumoto-Yor property/characterization and \cite{Piliszek} for the free version of the Hamza-Vallois property/characterization. In these two cases the transformations $\psi$ mapping independent/free $X$, $Y$ into $U$ and $V$ are involutions, i.e. $\psi=\psi^{-1}$ and thus the direct and dual properties are tautologically equivalent.  
\begin{theorem}
	Let $(\mathcal{A},\varphi)$ be a non-commutative probability space, let $X,Y\in\mathcal{A}$ be free both having the free Poisson distribution $\mu(\lambda,\alpha)$, $\mu(\kappa,\alpha)$, where $\lambda+\kappa>1$. Then random variables
	\begin{align*}
		U=(X+Y)^{-1/2}X(X+Y)^{-1/2}\qquad\mbox{and}\qquad V=X+Y
	\end{align*}
	are free.
\end{theorem}
\begin{proof}
	Due to the hypothesis $\lambda+\kappa>1$, and implicit assumption that $\varphi$ is a faithful trace, element $X+Y$ is invertible. 
	
	Let $(\mathcal{B},\phi)$ be a non-commutative probability space,  $U_1,V_1\in\mathcal{B}$ free with free binomial $(\lambda,\kappa)$ and free
	Poisson $(\lambda+\kappa,\alpha)$ distributions with respectively. It suffices to show that
	
	\begin{align}\label{eqn:*}
		\varphi\left(\prod_j U^{m_j} V^{n_j}\right)=\phi\left(\prod_j U_1^{m_j} V_1^{n_j}\right).
	\end{align}

	for any non-negative integers $m_j,n_j$.
	
	Define $X_1=V_1^{1/2}U_1V_1^{1/2}$, $Y_1=V_1-X_1$. From \cite[Th. 3.2]{SzpojanWesol} we get that $X_1$ and $Y_1$ are free with free Poisson distributions $\mu(\lambda,\alpha)$ and	$\mu(\kappa,\lambda)$ respectively. Thus $(X,Y)$ and $(X_1,Y_1)$ have the same joint distribution.
	
	But $\prod_j U^{m_j} V^{n_j}=g(X,Y)$ for some function $g$. Since $\lambda+\kappa>1$ the spectrum of $V$ is bounded away from zero and thus $U$ is in $C^*$-subalgebra generated by $X$ and $Y$. Consequently, it can be approximated by polynomials in free variables $X$ and $Y$. Therefore, the distribution of $g(X,Y)$ is uniquely determined by the distributions of $X$ and $Y$, see e.g. \cite{NicaSpeicherLect}.
	
	Of course, $\prod_j U_1^{m_j} V_1^{n_j}=g(X_1,Y_1)$ for the same function $g$.
	
	Since the quantity $\varphi(g(X,Y))$ for $X$ and $Y$ free depends only on distributions of $X$ and $Y$ so $\varphi(g(X,Y))=\phi(g(X_1,Y_1))$ and the proof is completed.
\end{proof}
\subsection*{Acknowledgement}
We would like to thank the anonymous referee for careful reading of the manuscript and several remarks which improved the presentation of our results.

\bibliographystyle{plain}
\bibliography{Bibl}

\end{document}